\documentclass[preprint,12pt]{elsarticle}
\usepackage{amssymb}
\usepackage{amsthm}
\journal{Elsevier}
\usepackage[dvipsnames]{xcolor}

\usepackage{amsmath}
\usepackage{amsfonts}
\usepackage{amssymb}
\usepackage{xcolor}
\usepackage{enumerate}
\usepackage{lineno}
\usepackage{tikz} \usetikzlibrary{calc} %para hacer cuentas con vectores (dibujos de arboles)

\newtheorem{definition}{Definition}[section]
\newtheorem{proposition}[definition]{Proposition}
\newtheorem{lemma}[definition]{Lemma}
\newtheorem{theorem}[definition]{Theorem}
\newtheorem{corollary}[definition]{Corollary}

\newdefinition{example}[definition]{Example}
\newdefinition{remark}[definition]{ Remark}
\newdefinition{problem}[definition]{ Problem}
\newdefinition{question}[definition]{Question}
\newdefinition{fact}[definition]{Fact}
\newproof{pot}{Proof}

\begin{document}

\begin{frontmatter}
\title{Uniqueness of the hyperspaces $C(p,X)$ in the class of trees}

\author{Florencio~Corona--V\'azquez  \corref{cor1}}
\ead{florencio.corona@unach.mx}

\author{Russell~Aar\'on~Qui\~nones--Estrella}
\ead{rusell.quinones@unach.mx}

\author{Javier~S\'anchez--Mart\'inez}
\ead{jsanchezm@unach.mx}

\author{Rosemberg Toal\'a--Enr\'iquez}
\ead{rosembergtoala@gmail.com}

\cortext[cor1]{Corresponding author}

\address{Universidad Aut\'onoma de Chiapas, Facultad de Ciencias en F\'isica y Matem\'aticas, Carretera Emiliano Zapata Km. 8, Rancho San Francisco, Ter\'an, C.P. 29050, Tuxtla Guti\'errez, Chiapas, M\'exico.}

\begin{abstract}
Given a continuum $X$ and $p\in X$, we will consider the hyperspace $C(p,X)$ of all subcontinua of $X$ containing $p$. Given a family of continua $\mathcal{C}$, a continuum $X\in\mathcal{C}$ and $p\in X$, we say that $(X,p)$ has unique hyperspace $C(p,X)$ relative to $\mathcal{C}$ if for each $Y\in\mathcal{C}$ and $q\in Y$ such that $C(p,X)$ and $C(q,Y)$ are homeomorphic, then there is an homeomorphism between $X$ and $Y$ sending $p$ to $q$. 
In this paper we study some topological and geometric properties about the structure of $C(p,X)$ when $X$ is a tree, being the main result that $(X,p)$ has unique hyperspace $C(p,X)$ relative to the class of trees. 
\end{abstract}


\end{frontmatter}
\section{Introduction}
A \textit{continuum} is a nonempty compact connected metric space. Given a continuum $X$, by a \textit{hyperspace} of $X$ we mean a specified collection of subsets of $X$. In the literature, some of the most studied hyperspaces are the following:   
\begin{align*}
2^{X} & = \{ A \subset X : A \text{ is nonempty and closed}\},        \\
C(X) & = \{ A \subset X : A \text{ is nonempty, connected and closed}\}, \\
F_n(X) & = \{A\subset X : A\textrm{ has at most $n$ points}\} \textrm{, where $n\in\mathbb{N}$}, \\
C(p,X) & = \{A\in C(X) : p\in A\} \textrm{, where $p\in X$}.
\end{align*}
% C(B,X) & = \{A\in C(X):B\subset A\} \qquad \textrm{ where $B\subset X$},\\

The collection $2^{X}$ is called the \textit{hyperspace of closed subsets} of $X$ whereas that $C(X)$ is called the \textit{hyperspace of subcontinua} of $X$. These hyperspaces are considered with the Hausdorff metric, see \cite[p. 1]{Nadler(1978)}.

Given a finite collection $K_{1},\dots,K_{r}$ of subsets of $X$, $\left\langle K_{1},\dots,K_{r}\right\rangle$, denotes the following subset of $2^X$: $$\left\{A\in 2^X:A\subset \displaystyle\bigcup^{r}_{i=1}K_{i}, A\cap K_{i}\neq\emptyset\textrm{ for each }i\in\left\{1,\dots,r\right\}\right\}.$$ It is known that the family of all subsets of $2^X$ of the form $\left\langle K_{1},\dots,K_{r}\right\rangle$, where each $K_{i}$ is an open subset of $X$, forms a basis for a topology for $2^X$ (see \cite[Theorem 0.11, p. 9]{Nadler(1978)}) called the \textit{Vietoris Topology}. The Vietoris topology and the topology induced by the Hausdorff metric coincide (see \cite[Theorem 0.13, p.~10]{Nadler(1978)}). The hyperspaces $C(X)$, $F_{n}(X)$ and $C(p,X)$ are considered as subspaces of $2^{X}$.

The topological structure of the hyperspaces $C(p,X)$ has been recently studied, for example, in \cite{CESV(2018)}, \cite{Eberhart(1978)}, \cite{Martinez(2007)}, \cite{Pellicer(2003)}, \cite{Pellicer(2005)}, and \cite{Pellicer(2005b)}, being useful to characterize classes of continua.

In general, given a hyperspace $\mathcal{H}(X)\in \{2^{X},C(X),F_{n}(X)\}$, we say that \textit{$X$ has unique hyperspace $\mathcal{H}(X)$} if for each continuum $Y$ such that $\mathcal{H}(X)$ is homeomorphic to $\mathcal{H}(Y)$, it holds that $X$ is homeomorphic to $Y$. This concept has served to characterize classes of continua through the structure of their hyperspaces, see for example \cite{Herrera(2007)}, \cite{Herrera(2009)}, \cite{Illanes(2003)} and \cite{Illanes(2013)}.
 
In a similar setting, given a family of continua $\mathcal{C}$, a continuum $X\in\mathcal{C}$ and $p\in X$, we say that $(X,p)$ \textit{has unique hyperspace $C(p,X)$ in (or relative to) $\mathcal{C}$} if for each $Y\in\mathcal{C}$ and $q\in Y$ such that $C(p,X)$ is homeomorphic to $C(q,Y)$, there is an homeomorphism $\mathfrak{h}:X\to Y$ such that $\mathfrak h(p)=q$ (compare with the continua \textit{$\mathcal{C}$-determined} given in \cite[p. 12]{Illanes(2002)} or \citep{Macias(1997)}). The main goal in this paper is to show that for each tree $X$ and 
$p\in X$, $(X,p)$ has unique hyperspace $C(p,X)$ in the class of trees (see Theorem \ref{principal}). 

The following examples show that the uniqueness of hyperspace $C(p,X)$ in the general setting is false, being our results the best that can be obtained in the class of trees. 

\begin{example}
In \cite[Lemma 3.19]{Pellicer(2003)} it was proved that a continuum $Y$ is hereditarily indecomposable if and only if for each $q\in Y$, $C(q,Y)$ is an arc. Also in \cite[Theorem 3.17]{Pellicer(2003)} it was proved that if $X$ is an arc and $p\in E(X)$, $C(p,X)$ is an arc. It follows that $(X,p)$ has no unique hyperspace in the class of continua.
\end{example}

\begin{example}
By using \cite[Observation 5.18 and Theorem 5.21]{Pellicer(2003)}, is easy to see that the Knaster continuum $Z$ with two end points $a$ and $b$ (see \cite[p. 205]{Kuratowski(1968)}) satisfy that $C(a,Z)$ and $C(b,Z)$ are arcs and $C(q,Z)$ is a 2-cell for each $q\in Z\setminus \{a,b\}$. This continuum shows that if $X$ is an arc and $p\in X$, then $(X,p)$ has no unique hyperspace in the class of continua.
\end{example}

\begin{example}\label{todoslosarboles}
If $X$ is a tree and $p\in X$, then there exist a continuum $Y$ and $q\in Y$, such that $Y$ is not a tree and $C(p,X)$ is homeomorphic to $C(q,Y)$. \\
To get such a continuum, let $Z$  be a Knaster continuum with two end points $a$ and $b$ and let $T$ be a simple triod with vertex $p\in T$. Choose an end point $x\in T$ and  attach $Z$ to $T$ by identifying $x$ with $a$. Let $Y$ be the space we get, $q:=[p]\in Y$. It is not difficult to see that $C(p,T)$ is homeomorphic to $C(q,Y)$. This shows that trees have no unique hyperspace $C(p,X)$ in the class of continua.
%To get such a continuum, let $Z$  be a Knaster continuum with two end points $a$ and $b$ and let $T$ be a tree with $p\in T$ a ramification point, choose an end point $x\in T$ and attach $Z$ to $T$ by identifying $a$ with $x$, getting the space $X$ with point $p=[x]=[a]$. It is not difficult to see that $C(p,X)=C(a,X)$.
\end{example}

\begin{example}
If $X$ is an arc with end points $a$ and $b$, and $Y$ is a simple closed curve, then $C(p,X)$ is homeomorphic to $C(q,Y)$ for each $p\in X\setminus \{a,b\}$ and $q\in Y$. This shows that if $p\in X\setminus \{a,b\}$ then $(X,p)$ has no unique hyperspace in the class of finite graphs.
\end{example}

\section{Definitions and preliminaries}

By a \textit{finite graph} we mean a continuum $X$ which can be written as the union of finitely many arcs, any two of which are either disjoint or intersect only in one or both of their end points. A \textit{tree} is a finite graph without simple closed curves. Given a positive integer $n$,  a \textit{simple $n$-od} is a finite graph, denoted by $T_n$, which is the union of $n$ arcs emanating from a single point, $v$, and  otherwise disjoint from each another. The point $v$ is called the \textit{vertex} of the simple $n$-od. A simple $3$-od, $T_{3}$, will be called a \textit{simple triod}.  A $n$\textit{-cell}, $I^{n}$, is any space homeomorphic to $[0,1]^n$. For a subset $A$ of $X$ we denote by $|A|$ the cardinality of $A$ and we use the symbols $\textrm{int}(A)$ and $\overline{A}$ to denotes the interior and the clousure of $A$ in $X$, respectively. 

Given a finite graph $X$, $p\in X$ and a positive integer $n$, we say that $p$ \textit{is of order $n$ in $X$}, denoted by 
ord$(p,X)=n$, if $p$ has a closed neighborhood which is homeomorphic to a simple $n$-od having $p$ as the vertex. 
If ord$(p,X)=1$ the point $p$ is called an \textit{end point of $X$}. The set of all end points of $X$ will be denoted by $E(X)$. If ord$(p,X)=2$ the point $p$ is called an \textit{ordinary point of $X$}. The set of all ordinary points of $X$ will be denoted by $O(X)$. A point $p\in X$ is a \textit{ramification point of $X$} if ord$(p,X)\geq 3$. The set of all ramification points of $X$ will be denoted by $R(X)$. The \textit{vertices} of a finite graph $X$ will be the end points and the ramification points of $X$, we denote by $V(X)$ the set of all vertices of $X$. An \textit{edge} will be an arc joining two elements of $V(X)$ and containing no more than two points of that set; the set of edges of $X$ will be denoted by $\mbox{edge}(X)$.

%In order to do this, we explain next some of the notation used from now on in the paper.\\
%
A \emph{subtree} $Y$ of the tree $X$ is a subcontinuum of $X$ such that if $e\in \mbox{edge}(Y)$ then $e\in \mbox{edge}(X)$, i.e., every edge in $Y$ should be an edge in $X$, it is avoid to use only a part of an edge of $X$ when constructing a subtree $Y$ of $X$.

In this paper, \textit{dimension} means inductive dimension as defined in \cite[(0.44), p. 21]{Nadler(1978)}. The symbol $\dim(X)$ will be used to denote the dimension of the space $X$.

%%%%%%%%%%%%%%%%%%%%
\section{About the cells in $C(p,X)$}
From now on $X$ will denote a tree and $p\in X$.\\ 

Given $n\in\mathbb{N}$, we denote by $U_n(X)$ the following set 
$$ \left\{ A\in C(p,X): \textrm{$A$ has a neighborhood in $C(p,X)$ homeomorphic to $I^n$}\right\},$$
we set $\mathcal{U} (X)=\bigcup \limits _{n\in\mathbb{N}}U_{n}(X)$, and as customary, $\pi _{0}(\mathcal U(X))$ denote the set of connected components of  $\mathcal U(X)$. 
 
Given a tree $X$, which is not an arc nor a simple $n$-od, we consider the  subcontinuum  
$$T(X)=\bigcup \left\{ e\in \textrm{edge($X$) : $e\cap E(X)=\emptyset$}\right\}.$$
For each $x\in E(X)$ there is exactly one point $r_x\in R(X)$ such that the
arc joining $x$ with $r_{x}$, $\overline{xr_x}$, is an edge of $X$. With this notation the continuum $T(X)$ can be described as 
the closure of
\[
 X\setminus \bigcup \limits _{x\in E(X)} \overline{xr_x}.
\]
The key idea is to remove all arcs incident in any end point of $X$ (see the picture below).

It is clear that the following statements hold.
\begin{itemize}
 \item $T(X)$ is a subtree of $X$ and $R(T(X))\subset R(X)\subset T(X)$.
 \item $x\in R(X)$ is stills a ramification point of $T(X)$ if and only if there is at least three different 
elements $y,z,w \in R(X)$ such that the arcs $\overline{yx}, \overline{zx}$ and $\overline{wx}$ are  edges of $X$.
\item For each $z\in E(T(X))$ holds 
\[
\left| \{ e\in \mbox{edge}(X)\setminus \mbox{edge}(T(X)): \text{$z$ is an end point of $e$}\} \right|\geq 2.
\]
\end{itemize}

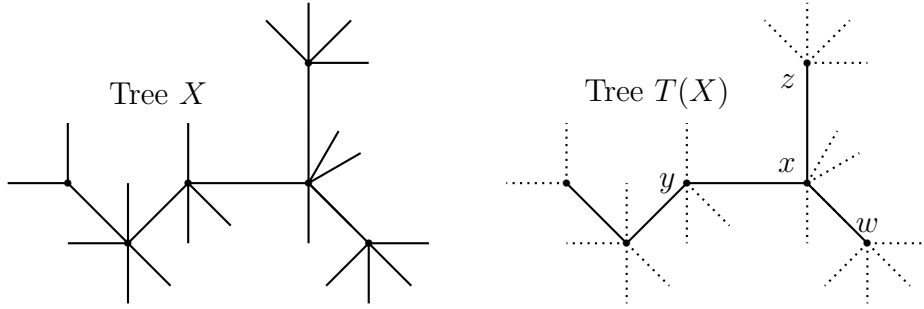
\[
\begin{tikzpicture}[thick, scale=0.8]
\node at (-0.5,1.5){Tree $X$};
\coordinate (a) at (0,0);
\coordinate (e) at (-1,-1);
\coordinate (b) at (0:2cm);
\coordinate (c) at ($(b)+(1,-1)$);
\coordinate (d) at ($(b)+(0,2)$);
\coordinate (f) at (180: 2cm);
\foreach \x in {a,b,c,d,e,f}{
  \draw node at (\x){\tiny $\bullet$};
}
\foreach \ang in {0,45,90,135}{
  \draw[shift=(d)] (0,0) -- (\ang : 1cm); 
  \draw[shift=(c)] (0,0) -- (-\ang : 1cm);
  }
\draw (a)--(b) (a)--(e) (b) -- (c) (b) --(d) (e) -- (f); 
\foreach \c in{a,e,f}{
    \draw[shift=(\c)] (0,0) -- (90 :1cm);
}
\foreach \c in{e,f}{
    \draw[shift=(\c)] (0,0) -- (180 :1cm);
}
\foreach \c in{a,b,e}{
    \draw[shift=(\c)] (0,0) -- (-90 :1cm) (0,0) -- (-45 :1cm);
    }
\draw[shift=(b)] (0,0) -- (30:1cm) (0,0) -- (60:1cm);
\end{tikzpicture}
\hspace{1cm}
%%%
\begin{tikzpicture}[thick, scale=0.8]
\node at (-0.5,1.5){Tree $T(X)$};
\coordinate (a) at (0,0);
\coordinate (e) at (-1,-1);
\coordinate (b) at (0:2cm);
\coordinate (c) at ($(b)+(1,-1)$);
\coordinate (d) at ($(b)+(0,2)$);
\coordinate (f) at (180: 2cm);
\draw[above left] node at (b){$x$};
\draw[left] node at (a){$y$};
\draw[below left] node at (d){$z$};
\draw[above] node at (c){$w$};
\draw (a)--(b) (a)--(e) (b) -- (c) (b) --(d) (e) -- (f); 
\foreach \x in {a,b,c,d,e,f}{
  \draw node at (\x){\tiny $\bullet$};
}
\begin{scope}[dotted]
\foreach \ang in {0,45,90,135}{
  \draw[shift=(d)] (0,0) -- (\ang : 1cm); 
  \draw[shift=(c)] (0,0) -- (-\ang : 1cm);
  }
\foreach \c in{a,e,f}{
    \draw[shift=(\c)] (0,0) -- (90 :1cm);
}
\foreach \c in{e,f}{
    \draw[shift=(\c)] (0,0) -- (180 :1cm);
}
\foreach \c in{a,b,e}{
    \draw[shift=(\c)] (0,0) -- (-90 :1cm) (0,0) -- (-45 :1cm);
    }
\draw[shift=(b)] (0,0) -- (30:1cm) (0,0) -- (60:1cm);
\end{scope}
\end{tikzpicture}
\]
Let $\mbox{Sub}_{p}(T(X))$ the collection of all subtrees of $T(X)$ containing $p$. One of the main goals of this section is to establish a one to one corres\-pondence between $\mbox{Sub}_p(T(X))$ and the components of $\mathcal U(X)$.  In order to achieve this, let $Y\in \mbox{Sub}_{p}(T(X))$ and for each $e\in \mbox{edge}(X)\setminus \mbox{edge}(Y)$ such that $e\cap Y\neq \emptyset$, the intersection $e\cap Y$ consist of exactly one point, so we can pick a homeomorphism  $h_{e}: [0,1]\rightarrow e$ in such a way that $h_{e}(0)\in Y$.\\

We define the following set of $C(p,X)$
%\[
%U_{Y}:= \left \{  Y\cup \!\!\! \bigcup \limits_{\tiny \begin{matrix} e\in \mbox{edge}(X)\setminus \mbox{edge}(Y) \\ e\cap Y\neq \emptyset\end{matrix}} \!\! h_{e}(t) : t\in (0,
%\right\}
%\]

\[
U_{Y}:= \left\{ Y\cup \bigcup h_{e}(t_{e}):
e\in \mbox{edge}(X)\setminus \mbox{edge}(Y),\ e\cap Y\neq \emptyset,\ t_{e}\in (0,1) \right\}.
\]
Observe that $U_{Y}\subset C(Y,X):=\{A\in C(X): Y\subset A\}$. On the other hand, notice that if the set $\{e\in \mbox{edge}(X)\setminus \mbox{edge}(Y),\ e\cap Y\neq \emptyset\}$ is numbered as $\{e_{1},\dots, e_{k}\}$ then $U_{Y}=\{Y\cup A \in C(X):A\in \langle \textrm{int}(e_{1}),\dots,\textrm{int}(e_{k})\rangle\}$, thus $U_{Y}$ is homeomorphic to $\prod^{k}_{i=1}\textrm{int}(e_{i})$; these two facts shows the following result.

%%%%%%%%%%%%%%%%
\begin{proposition}\label{components}
Let $Y\in \textup{Sub}_{p}(T(X))$, $A\in U_{Y}$ and $n=|E(A)|$. Then
\begin{enumerate}[a)]
	\item $n=|\{ e\in \textup{edge}(X)\setminus \textup{edge}(Y):e\cap Y\neq \emptyset\}|$ and $U_{Y}\subset U_{n}(X)$.
	\item For each $B\in U_{Y}$, $Y=\bigcup \{A\in \textup{Sub}_{p}(T(X)):A\subset B\}$, consequently, $Y$ is the largest element in $\textup{Sub}_{p}(T(X))$ contained in $B$.
	\item If $Y'\in \textup{Sub}_{p}(T(X))$ is such that $U_{Y}\cap U_{Y'}\neq \emptyset$, then $Y=Y'$.
	\item $U_{Y}$ is open and connected in $\mathcal U(X)$.
\end{enumerate}
\end{proposition}
Next we establish some consequences of the previous proposition.

According to (c) and (d), given $Y\in \textup{Sub}_{p}(T(X))$, $U_{Y}$ is a component of $\mathcal U(X)$, therefore we get a function
$\phi_X: \textup{Sub}_{p}(T(X))\to \pi_{0}(\mathcal U(X))$, given by $\phi_{X}(Y)=U_Y$.

\begin{corollary}
The function $\phi_{X}$ is bijective.
\end{corollary}
\begin{proof}
By (c) from Proposition \ref{components}, we have that $\phi_X$ is one to one. In order to prove that $\phi_X$ is surjective, it is enough to see that for each $B\in \mathcal U(X)$ there exists $Y\in \textup{Sub}_{p}(T(X))$ such that $B\in U_Y$. Let $Y=\bigcup \{A\in \textup{Sub}_{p}(T(X)):A\subset B\}$. It is clear that $Y\in C(p,X)\cap \textup{Sub}_{p}(T(X))$ and $B\in U_Y$.
\end{proof}

%Since $U_{Y}$ is a component of $\mathcal U(X)$ there exist a uniquely determined $n\in \mathbb{N}$ such that $U_{Y}\subset U_{n}$, this enable us to speak about the dimension of $U_{Y}$, defined as the $n$ above, i.e., $\dim (U_{Y})=n$. With this notation we have the following monotony property of the bijection $\phi$.

Since $U_{Y}$ is a component of $\mathcal U(X)$ there exists $n\in \mathbb{N}$ such that $U_{Y}\subset U_{n}(X)$, thus $\dim (U_{Y})=n$. We have the following monotony property of $\phi_{X}$.

\begin{corollary}
Let $Y,Y'\in \textup{Sub}_{p}(T(X))$. If $Y\subset Y'$ then $\dim (\phi_X (Y))\leq \dim (\phi_X (Y'))$.
\end{corollary}
\begin{proof}
Using (a) from Proposition \ref{components}, we get 
\begin{align*}
\dim U_{Y} &= |\{ e\in \textup{edge}(X)\setminus \textup{edge}(Y):e\cap Y\neq \emptyset\}|\\
&\leq  |\{ e\in \textup{edge}(X)\setminus \textup{edge}(Y'):e\cap Y'\neq \emptyset\}|= \dim U_{Y'}.
\end{align*}

\end{proof}

\begin{corollary}\label{minimax}
The following holds
\begin{align*}
\textup{ord}(p,X) & = \min \{ \dim (U_Y): Y\in \textup{Sub}_p(T(X))\} = \dim U_{\{p\}}, \\
|E(X)| & = \max \{ \dim (U_Y): Y\in \textup{Sub}_p(T(X))\} = \dim U_{T(X)}.
\end{align*}

Moreover $U_{\textup{ord}(p,X)}(X)=U_{\{p\}}$ and $U_{|E(X)|}(X)=U_{T(X)}$.
%The cell of smallest dimension in $\mathcal U(X)\subset C(p,X)$ is of dimension $\textup{ord}(p,X)$ and the cell of largest dimension is of dimension $|E(X)|$.
\end{corollary}

\begin{proof}
The right-hand side equalities are trivial from previous corollary. The left-hand side equalities are also an easy consequence of (a) from Proposition \ref{components}.
%It is enough to see that $\{p\}, T(X)\in \mbox{Sub}_{p}(T(X))$ and that $\%{p\} \subset Y\subset  T(X)$ for every $Y\in \mbox{Sub}_{p}(T(X))$.
\end{proof}

\section{Unique hyperspace $C(p,X)$ of trees}

We will denote by $\mathcal{T}$ the class of trees, by $\mathcal{I}$ the class of arcs, by $\mathcal{N}$ the class of simple $n$-ods and by $\widehat{\mathcal{T}}=\mathcal{T}\setminus (\mathcal{I}\cup\mathcal{N})$.

The following result is a slight modification of \cite[Corollary 3.4, p. 45]{CESV(2018)}, and its proof is essentially the same.

\begin{proposition}\label{ordenes}
Let $X,Y\in\mathcal{T}$ and $p\in X$, $q\in Y$. If $C(p,X)$ is homeomorphic to $C(q,Y)$, then $\textup{ord}(p,X)=\textup{ord}(q,Y)$.
\end{proposition}

\begin{definition}
Let $n\in \mathbb{N}$. A continuum $Y$ is an \textit{$n$-od} if there exists $K\in C(Y)$ such that $Y\setminus K$ has at least $n$ components. Further, we will say that $K$ is a \textit{core} of the $n$-od. If $n=3$, $Y$ is called a \textit{triod}.
\end{definition}

\begin{theorem}\label{arcos}
If $X\in \mathcal{I}$ and $p\in X$, then $(X,p)$ has unique hyperspace in $\mathcal{T}$.
\end{theorem}

\begin{proof}
If $X$ is an arc with end points $a$ and $b$ and $p\in X$, then $C(p,X)$ is an arc if $p\in\{a,b\}$ and $C(p,X)$ is a two cell if $p\in X\setminus \{a,b\}$ (see \cite[Theorem 3.17, p. 265]{Pellicer(2003)}). By \cite[Lemma 3.15, p. 264]{Pellicer(2003)}, if $Y$ is a continuum and $q\in Y$ is such that $C(q,Y)$ is an arc or a two cell, then $q$ is not in a core of a triod in $Y$, thus in the case that $Y\in \mathcal{T}$, $Y$ must be an arc. By Proposition \ref{ordenes}, is easy to see that $C(p,X)$ homeomorphic to $C(q,Y)$ implies that there exists an homeomorphism $h:X\to Y$ sending $p$ to $q$.
\end{proof}

\begin{theorem}\label{nodos}
If $X\in \mathcal{N}$ and $p\in X$, then $(X,p)$ has unique hyperspace in $\mathcal{T}$.
\end{theorem}

\begin{proof}
Let $p\in X$ and suppose that $Y\in \mathcal{T}$ and $q\in Y$ is such that $C(p,X)$ is homeomorphic to $C(q,Y)$. We consider the following two cases.\\
Case 1. If $p\in R(X)$. By Proposition \ref{ordenes}, $q\in R(Y)$ and $\textrm{ord}(p,X)=\textrm{ord}(q,Y)$. If $n=\textrm{ord}(p,X)$, then $C(p,X)$ and $C(q,Y)$ are $n$-cells, thus $Y$ is a simple $n$-od with vertex $q$.\\  
Case 2. If $q\in E(X)\cup O(X)$. By Proposition \ref{ordenes}, $q\in E(Y)\cup O(X)$. Consider $l_{p}$ and $l_{q}$ the edges in $X$ and $Y$, containing $p$ and $q$, respectively. Let $x\in (V(X)\setminus \{p\})\cap l_{p}$ and $y\in (V(Y)\setminus \{q\})\cap l_{q}$. By \cite[Proposition 3.7, p. 45]{CESV(2018)}, we have that $\textrm{ord}(x,X)=\textrm{ord}(y\in Y)$. If $Y$ contains  a point $r\in R(Y)\setminus \{y\}$, then $C(q,Y)$ contains an $(\textrm{ord}(q,Y)+\textrm{ord}(r,Y))$-cell, but this in impossible because $\dim(C(q,Y))=\dim(C(p,X))=\textrm{ord}(x,X)=\textrm{ord}(y,Y)$. Therefore, $Y$ is an simple $\textrm{ord}(x,X)$-od.\\
In both cases the existence of an homeomorphism $h:X\to Y$ sending $p$ to $q$ is trivial.
\end{proof}

%\begin{proposition}
%Let $X,Y\in \widehat{\mathcal{T}}$, $p\in X$ and $q\in Y$. If 
%$h:C(p,X)\to C(q,Y)$ is an homeomorphism then $h(C(T(X),X))\subset C(T(Y),Y)$. 
%\end{proposition}
%
%Given a tree $X$ non an arc and a point $p\in R(X)$, we define recursively a tree $X_{n}$ as follows:
%\begin{itemize}
%\item $X_{1}=\bigcup\{A: A\textrm{ is an edge of $X$ containing $p$}\}$,
%\item if $X_{n}\neq X$, then $$X_{n+1}=X_{n}\cup \bigcup\{A: A\textrm{ is an edge of $X$ and $A\cap E(X_{n})\neq \emptyset$}\}.$$
%\end{itemize}
%
%This gives a finite filtration of trees:
%\[
%X_1\subset X_2 \subset X_3 \subset \cdots \subset X_N=X,
%\]
%where $X_1$ is a simple $n$-od for $n=\mbox{ord}(p,X)$. Observe that $X_{m+1}$ is obtained from $X_m$ by 
%adding some edges of $X$, each of them intersecting $E(X_{m+1})$, hence we get the following result. 
%
%\begin{proposition}
%If $X\in \widehat{\mathcal{T}}$, $p\in R(X)$ and is defined $X_{m+1}$ then $T(X_{m+1})\subset X_{m}$. 
%\end{proposition}

%\begin{theorem}\label{ramificacion}
%If $X\in \mathcal{T}-\mathcal{I}$ and $p\in R(X)$, then $(X,p)$ has unique hyperspace $C(p,X)$ in $\mathcal{T}$. Moreover, if $C(p,X)$ is homeomorphic to $C(q,Y)$, where $Y\in \mathcal{T}$ and $q\in Y$, there exists an homeomorphism between $X$ and $Y$ sending $p$ in $q$.
%\end{theorem}

\textit{Notation.} Let $X\in \widehat{\mathcal{T}}$. Suppose that $G\in \textup{Sub}_p(T(X))$ consists of the edges $e_1, \dots, e_k$ and let $f_1, \ldots, f_m$ be the edges of $X\setminus G $ incident on $G$. We will use the notation 
	$$U_G = [G;\, f_1, \ldots, f_m]= \left[ e_1, \ldots, e_k \, ; \, f_1, \ldots, f_m \right],$$
to describe the component $\phi(G)$ of $\mathcal{U}(X)$. 

We will need the following lemma:

\begin{lemma} \label{incidencias}
Let $X\in \widehat{\mathcal{T}}$. Let $G=e_1\cup \ldots \cup e_k \subset T(X)$ be a collection of edges and let $f_1, \ldots, f_m$ be the edges of $X\setminus G $ incident on $G$. Then $m\geq 2$.
\end{lemma}

\begin{proof}
By induction on $k$. For $k=1$, let $e=\overline{p_1p_2}$, we can assume $p_2\neq p$. Then ord$(p_2,X)\geq 3$, since all the vertices $T(X)$ are ramification points; in other words, there are at least two incident edges on $e=\overline{p_1p_2}$. 

Next, assume the lemma holds for $k$ and consider $G=e_1\cup \ldots \cup e_k\cup e_{k+1}$, since $T(X)$ is a tree we can choose, without loss of generality, $e_{k+1}$ such that it intersects at most one of $e_1, \ldots, e_k$. Let $e_{k+1}=f_1, \ldots, f_m$ be the edges incident on $G$, by hypothesis  $m\geq 2$. Furthermore let the point $p_2$ be the extreme of $e_{k+1}$ which is not in $G$, once again we have ord$(p_2,X)\geq 3$, that is, the edges incident on $G'=e_1\cup \ldots \cup e_k \cup e_{k+1}$ are $m-1 + \textrm{ord}(p_2,X)-1\geq m+1 \geq 2$. This finishes the induction and the proof of the lemma.

\end{proof}

\begin{lemma} \label{no-interseccion celdas}
Let $X\in \widehat{\mathcal{T}}$. Let $G, G'\in \textup{Sub}_p(T(X))$, suppose there is an $e \in \textup{edge}(G)$ such that $e \cap G' = \emptyset$ (i.e. $e$ is not an edge of $G'$ nor it is incident on $G'$). Then $\overline{U}_G \cap \overline{U}_{G'} = \emptyset $.
\end{lemma}

\begin{proof}
Let $f_1, \ldots , f_m$ be the edges of $X\setminus G$ incident on $G$, similarly, let $f'_1, \ldots , f'_{m'}$  be the edges of $X\setminus G'$ incident on $G'$. Note that an element $K'\in \overline{U}_{G'}$ is of the form $K'=G' \cup \bigcup^{m^{\prime}}_{i=1}\{A_{i}:A_{i}\in C(f^{\prime}_{i})\}$, similarly an element $K\in \overline{U}_{G}$ is of the form $K=G \cup \bigcup^{m}_{j=1}\{A_{j}:A_{j}\in C(f_{j})\}$, therefore $e \subset K$, however there is no way $e\subset K'$, since $e$ is not an edge of $G'$ nor it is incident on $G'$.  Hence $\overline{U}_G $ and $ \overline{U}_{G'}$ cannot have elements in common. 
\end{proof}

The case when all the edges of $G$ are edges of $G'$ or are incident on $G'$ and vice versa is covered in the following:

\begin{proposition}
Let $X\in \widehat{\mathcal{T}}$. Consider $G$, $G' \in \textup{Sub}_p(T(X))$ and let
	\begin{itemize}
	\item $e_1, \ldots ,e_k$ be the edges of $G \cap G'$.
	\item $e_{k+1}, \ldots, e_{k+l}$ be the edges of $G$ which are not in $G'$.
	\item $e'_{k+1}, \ldots, e'_{k+l'}$ be the edges of $G'$ which are not in $G$.
	\item $d_1, \ldots, d_n$ be the edges of $X\setminus(G\cup G')$ incident on $G\cap G'$.
	\item $f_1, \ldots, f_m$ be the edges of $X\setminus(G\cup G')$ incident on $G$ but not on $G'$.	
	\item $f'_1, \ldots, f'_{m'}$ be the edges of $X\setminus(G\cup G')$ incident on $G'$ but not on $G$.
	\end{itemize}
 Suppose that $e'_{k+1}, \ldots, e'_{k+l'}$ are incident on $G$ and $e_{k+1}, \ldots, e_{k+l}$ are incident on $G'$. Then
	\begin{enumerate}
		\item  $\overline{U}_G \cap \overline{U}_{G'} \neq \emptyset $.
		\item $\dim(\overline{U}_G \cap \overline{U}_{G'}) = n$.
	\end{enumerate}
\end{proposition}

	\begin{proof}
Note that, $G$ consists of the edges $e_1, \ldots, e_k , \ldots e_{k+l}$ and the edges incident on $G$ are $d_1, \ldots, d_n, e'_{k+1} \ldots e'_{k+l'}, f'_1, \ldots, f'_{m'}$. Similarly for $G'$, thus,
	\begin{align*}
		U_G & =  \left[ e_1, \ldots, e_k , e_{k+1}, \ldots, e_{k+l} \,  ; \,   d_1, \ldots, d_n, e'_{k+1}, \ldots, e'_{k+l'},  f_1, \ldots, f_{m}  \right],	\\
		U_{G'} & =  \left[ e'_1, \ldots, e'_k , e'_{k+1}, \ldots, e'_{k+l'}  \, ; \,   d_1, \ldots, d_n, e_{k+1}, \ldots, e_{k+l}, f'_1, \ldots, f'_{m'}  \right].
	 	\end{align*}
The following pictures should make clear the above notation.
	 	
	 	
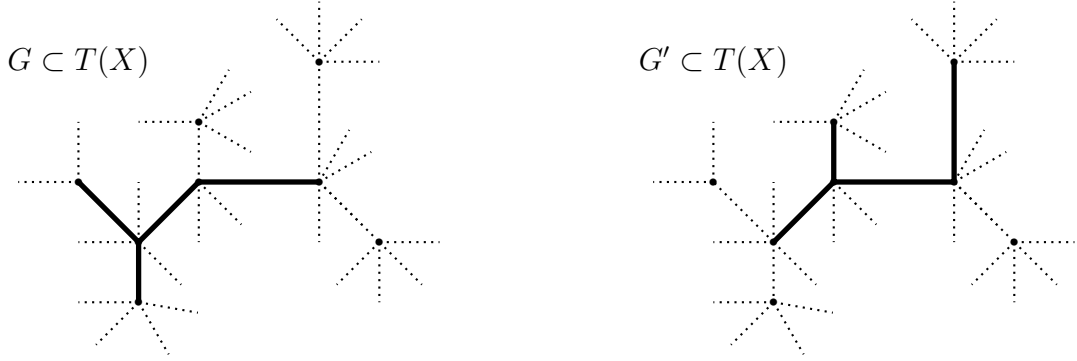
\[
 \begin{tikzpicture}[thick, scale=0.8]
\coordinate (a) at (0,0);
\coordinate (e) at (-1,-1);
\coordinate (b) at (0:2cm);
\coordinate (c) at ($(b)+(1,-1)$);
\coordinate (d) at ($(b)+(0,2)$);
\coordinate (f) at (180: 2cm);
\coordinate (n1) at (-1,-2); \coordinate (n2) at (0,1);
%%%%%%55%%%%%%%%%%%%%%%%%5
%%%
\foreach \x in {a,b,c,d,e,f,n1,n2}{
  \draw node at (\x){\tiny $\bullet$};
}
%%%%%%%%%%%%%%%5
\begin{scope}[dotted]
%%%%%
\draw (a)--(b) (a)--(e) (b) -- (c) (b) --(d) (e) -- (f); 
\foreach \ang in {0,45,90,135}{
  \draw[shift=(d)] (0,0) -- (\ang : 1cm); 
  \draw[shift=(c)] (0,0) -- (-\ang : 1cm);
  }
\foreach \c in{a,e,f}{
    \draw[shift=(\c)] (0,0) -- (90 :1cm);
}
\foreach \c in{e,f}{
    \draw[shift=(\c)] (0,0) -- (180 :1cm);
}
\foreach \c in{a,b,e}{
    \draw[shift=(\c)] (0,0) -- (-90 :1cm) (0,0) -- (-45 :1cm);
    }
\draw[shift=(b)] (0,0) -- (30:1cm) (0,0) -- (60:1cm);
\foreach \c in{180,60,30,-30}{
  \draw[yshift=1cm] (0,0) -- (\c:1cm);
}
\foreach \c in{180,225,300,350}{
  \draw[xshift=-1cm, yshift=-2cm] (0,0) -- (\c: 1cm);
}
\end{scope}
\draw[line width=2pt] (f) -- (e) -- (n1) (e)--(a)-- (b);
\node at (-2,2){$G\subset T(X)$};
\end{tikzpicture}
%%%%%%%%%%%%%
\hspace{2.5cm}
%%%
 \begin{tikzpicture}[thick, scale=0.8]
\coordinate (a) at (0,0);
\coordinate (e) at (-1,-1);
\coordinate (b) at (0:2cm);
\coordinate (c) at ($(b)+(1,-1)$);
\coordinate (d) at ($(b)+(0,2)$);
\coordinate (f) at (180: 2cm);
\coordinate (n1) at (-1,-2); \coordinate (n2) at (0,1);
%%%%%%55%%%%%%%%%%%%%%%%%5
%%%
\foreach \x in {a,b,c,d,e,f,n1,n2}{
  \draw node at (\x){\tiny $\bullet$};
}
%%%%%%%%%%%%%%%5
\begin{scope}[dotted]
%%%%%
\draw (a)--(b) (a)--(e) (b) -- (c) (b) --(d) (e) -- (f); 
\foreach \ang in {0,45,90,135}{
  \draw[shift=(d)] (0,0) -- (\ang : 1cm); 
  \draw[shift=(c)] (0,0) -- (-\ang : 1cm);
  }
\foreach \c in{a,e,f}{
    \draw[shift=(\c)] (0,0) -- (90 :1cm);
}
\foreach \c in{e,f}{
    \draw[shift=(\c)] (0,0) -- (180 :1cm);
}
\foreach \c in{a,b,e}{
    \draw[shift=(\c)] (0,0) -- (-90 :1cm) (0,0) -- (-45 :1cm);
    }
\draw[shift=(b)] (0,0) -- (30:1cm) (0,0) -- (60:1cm);
\foreach \c in{180,60,30,-30}{
  \draw[yshift=1cm] (0,0) -- (\c:1cm);
}
\foreach \c in{180,225,300,350}{
  \draw[xshift=-1cm, yshift=-2cm] (0,0) -- (\c: 1cm);
}
\end{scope}
\draw[line width=2pt] (e) -- (a) -- (n2) (a)--(b)-- (d);
\node at (-2,2){$G'\subset T(X)$};
\end{tikzpicture}
\]
%%%%%%%%%%%%%%%%%%%%%%%%%%%%%%%%%%%%%%%%%%%%%%%%%%%%%%%%%%%%%%%%%%%%%%%%%%%%%%%%%
In the case of the above subtrees, the edges are as follows:
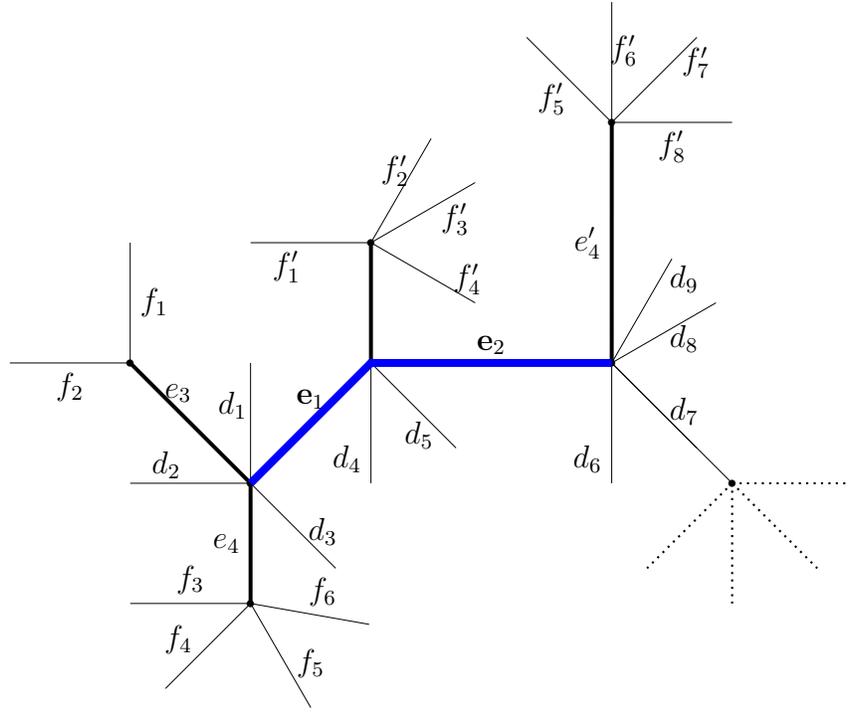
\[
  \begin{tikzpicture}[thick, scale=1.6]
\coordinate (a) at (0,0);
\coordinate (e) at (-1,-1);
\coordinate (b) at (0:2cm);
\coordinate (c) at ($(b)+(1,-1)$);
\coordinate (d) at ($(b)+(0,2)$);
\coordinate (f) at (180: 2cm);
\coordinate (n1) at (-1,-2); \coordinate (n2) at (0,1);
%%%%%%55%%%%%%%%%%%%%%%%%5
%%%
\foreach \x in {a,b,c,d,e,f,n1,n2}{
  \draw node at (\x){\tiny $\bullet$};
}
%%%%%%%%%%%%%%%5
\begin{scope}[very thin]
%%%%%
\draw (a)--(b) (a)--(e) (b) -- (c) (b) --(d) (e) -- (f); 
\foreach \ang in {0,45,90,135}{
  \draw[shift=(d)] (0,0) -- (\ang : 1cm); 
%  \draw[shift=(c)] (0,0) -- (-\ang : 1cm);
  }
\foreach \c in{a,e,f}{
    \draw[shift=(\c)] (0,0) -- (90 :1cm);
}
\foreach \c in{e,f}{
    \draw[shift=(\c)] (0,0) -- (180 :1cm);
}
\foreach \c in{a,b,e}{
    \draw[shift=(\c)] (0,0) -- (-90 :1cm) (0,0) -- (-45 :1cm);
    }
\draw[shift=(b)] (0,0) -- (30:1cm) (0,0) -- (60:1cm);
\foreach \c in{180,60,30,-30}{
  \draw[yshift=1cm] (0,0) -- (\c:1cm);
}
\foreach \c in{180,225,300,350}{
  \draw[xshift=-1cm, yshift=-2cm] (0,0) -- (\c: 1cm);
}
\end{scope}
\foreach \ang in {0,45,90,135}{
%  \draw[shift=(d)] (0,0) -- (\ang : 1cm); 
  \draw[shift=(c), dotted] (0,0) -- (-\ang : 1cm);
  }
\draw[line width=1.5pt] (e) -- (a) -- (n2) (a)--(b)-- (d);
\draw[line width=1.5pt] (f) -- (e) -- (n1);
\draw[line width=3pt, blue] (e)--(a)-- (b);
%etiquetas:
\node at (-0.5,-0.3){$\mathbf e_1$}; \node at (1,0.15){$\mathbf e_2$}; 
\node at (-1.6,-0.25){$e_3$}; \node at (-1.2,-1.5){$e_4$}; 
\node at (1.8,1){$e'_4$}; 
% f-sss
\foreach \x/\y\i in{-1.8/0.5/1,-2.5/-0.2/2,-1.5/-1.8/3,-1.6/-2.3/4,-0.5/-2.5/5, -0.4/-1.9/6}{
%2.2/-1.5/7,3.2/-1.7/8,3.7/-1.5/9, 3.5/-0.8/10}{
    \node at (\x,\y) {$f_{\i}$};
}
%d-ssss
\foreach \x/\y\i in{-1.15/-0.35/1, -1.7/-0.85/2, -0.4/-1.4/3, -0.2/-0.8/4, 0.4/-0.6/5,1.8/-0.8/6,
2.6/-0.4/7, 2.6/0.2/8,2.6/0.7/9}{
    \node at (\x,\y) {$d_{\i}$};
}
%f'-ssss
\foreach \x/\y/\i in{-0.7/0.8/1,0.2/1.6/2,0.7/1.2/3, 0.8/0.7/4,
1.5/2.2/5,2.1/2.6/6,2.7/2.5/7, 2.5/1.8/8}{
    \node at (\x,\y) {$f'_{\i}$};
}
\end{tikzpicture}
\]

	Therefore, $\dim(U_G)$= $n+l'+m$ and $\dim(U_{G'})$= $n+l+m'$. Moreover, it can be seen from this description that $\overline{U}_G \cap \overline{U}_{G'}$ consists of elements of the form $K = G\cup G' \cup \bigcup^{n}_{k=1}\{A_{k}:A_{k}\in C(d_{k})\}$. Hence $\dim(\overline{U}_G \cap \overline{U}_{G'}) = n $.	
	\end{proof}
	 	 
As an important consequence of the description given above, we have:

\begin{proposition} \label{interseccion celdas}
Let $X\in \widehat{\mathcal{T}}$. Consider $G$, $G' \in \textup{Sub}_p(T(X))$ such that $G'=G \cup e$, for some $e\in \textup{edge}(T(X))$. Then 
	\begin{enumerate}
		\item  $\overline{U}_G \cap \overline{U}_{G'} \neq \emptyset $.
		\item $\dim(\overline{U}_G \cap \overline{U}_{G'}) = \dim(U_G)-1$.
		\item $\dim(U_{G'})= \dim(U_{G}) + \textup{ord}(p',X) - 2$, where $p'$ is the vertex of $e$ which is not in $G$.
	\end{enumerate}
	And vice versa, if $\overline{U}_G \cap \overline{U}_{G'} \neq \emptyset $ and $\dim(\overline{U}_G \cap \overline{U}_{G'}) = \dim(U_G)-1$, then there exists an unique 
	$e\in \textup{edge}(T(X))$ such that $G'=G\cup e$.
\end{proposition}

\begin{proof}
 Using the notation of the previous proposition we have in this case that $G=G\cap G'=\{e_1, \ldots, e_k\}$, $e=e'_{k+1}$ is the only edge of $G'$ incident on $G$, $d_1, \ldots, d_n$ are the edges incident on $G$, there are no $f$'s and $e, f'_1, \ldots, f'_{m'}$ are the edges incident on $p'$. Hence,
\begin{align*}
		U_G & =  \left[ e_1, \ldots, e_k \,  ; \,   d_1, \ldots, d_n, e \right],	\\
		U_{G'} & =  \left[ e_1, \ldots, e_k , e \, ; \,   d_1, \ldots, d_n, f'_1, \ldots, f'_{m'}  \right].
	 	\end{align*}
From the previous it can be seen that $\overline{U}_G \cap \overline{U}_{G'}$ consists of elements of the form $K=G' \cup \bigcup^{n}_{i=1}\{A_{i}:A_{i}\in C(d_{i})\}$. Thus, $\overline{U}_G \cap \overline{U}_{G'} \neq \emptyset $ and $\dim(\overline{U}_G \cap \overline{U}_{G'}) = n$. 

Moreover, 
\begin{align*}
\dim(U_G) &= n+1, \\
\dim(U_{G'}) &= n+m'= \dim(U_G)-1 + \textrm{ord}(p',X)-1.
\end{align*} 
This concludes the first part.

For the converse, first observe that by Lemma \ref{no-interseccion celdas} and the condition $\overline{U}_G \cap \overline{U}_{G'} \neq \emptyset $, all the edges of $G\setminus G'$ are incident on $G'$ and all the edges of $G'\setminus G$ are incident on $G$; thus satisfying the condition of the previous proposition. Then the condition $\dim(\overline{U}_G \cap \overline{U}_{G'}) = \dim(U_G)-1$ translates to $n=n+l'+m-1$, that is, $l'+m=1$. We only have two possibilities $(l', m)=(1,0)$ or $(l', m)=(0,1)$. The first case is precisely when $e=e'_{k+1}$ is the only edge of $G'$ which is not in $G$, and $m=0$ implies that $G\subset G'$. The second case, when $m=1$, is not possible due to Lemma \ref{incidencias}. Therefore the only possibility is $(l', m)=(1,0)$, so $G'=G\cup e$. Since $G'$ is given, then the edge $e$ is unique.
\end{proof}

\begin{definition}
Let $p,p'\in X$ be vertices in $X\in\mathcal{T}$, we will denote by $\overline{pp'}$ the smallest connected subtree containing $p$ and $p'$.
\end{definition}

\begin{proposition} \label{bijection}
Let $X, Y \in \mathcal{T}$, $p \in X$ and $q \in Y$. If $h : C(p, X) \rightarrow C(q,Y)$ is an homeomorphism then for each vertex $p' \in T(X)$ there exists a unique $q'\in T(Y)$ such that $h(U_{\overline{pp'}})=U_{\overline{qq'}}$. Moreover, if $\overline{pp'}$ consists of the edges $e_i=\overline{p_{i-1}p_i}$, $i=1, \ldots, n$, with $p=p_0$ and $\overline{qq'}$ consists of the edges $f_i=\overline{q_{i-1}q_i}$, $i=1, \ldots, m$, with $q=q_0$, then $n=m$ and $h(U_{\overline{pp_i}})=U_{\overline{qq_i}}$ for $i=0, 1, \ldots, n$.
\end{proposition}

\begin{proof}
By induction on $n$. The case $n=0$ follows straightforward from Corollary \ref{minimax}. The case $n=1$ is covered by Proposition \ref{interseccion celdas}, more precisely, consider an edge $e=\overline{pp'} \subset T(X)$ and let $G=\{p\}$ and $G'=G\cup e$, then $\overline{U}_G \cap \overline{U}_{G'} \neq \emptyset $ and  $\dim(\overline{U}_G \cap \overline{U}_{G'}) = \dim(U_G)-1$. These conditions are preserved under the homeomorphism $h$, that is, $h(\overline{U}_G) \cap h(\overline{U}_{G'}) \neq \emptyset $ and  $\dim(h(\overline{U}_G) \cap h(\overline{U}_{G'})) = \dim(h(U_G))-1$ and we know $h(U_G)=U_{\{q\}}$. Since $\phi_{Y}$ is bijective there exists $F'\in  \textup{Sub}_p(T(Y))$ such that $\phi_{Y}(F')=U_{F'}=h(U_{G'})$. By the second part of Proposition \ref{interseccion celdas} applied to $U_{\{q\}}$ and $U_{F'}$, there exists an unique edge $f'\in \textup{edge}(T(Y))$ such that $F'=\{q\}\cup f'$, in other words, there is a unique vertex $q'\in T(X)$ and an edge $f'=\overline{qq'}\subset T(Y)$ such that $h(U_{\overline{pp'}})=U_{\overline{qq'}}$.

Now suppose the proposition holds for $n$ and consider the subtree $\overline{pp'}$ having edges $e_{i}=\overline{p_{i-1}p_i}$, $i=1,2,\ldots,n+1$. Once again we rely completely on Proposition \ref{interseccion celdas}, let $G=e_1\cup \ldots \cup e_{n}$, $G'=G\cup e_{n+1}$, then $\overline{U}_G \cap \overline{U}_{G'} \neq \emptyset $ and  $\dim(\overline{U}_G \cap \overline{U}_{G'}) = \dim(U_G)-1$. These conditions are preserved under the homeomorphism $h$, that is, $h(\overline{U}_G) \cap h(\overline{U}_{G'}) \neq \emptyset $ and  $\dim(h(\overline{U}_G) \cap h(\overline{U}_{G'})) = \dim(h(U_G))-1$ and we know by hypothesis of induction that $h(U_{\overline{pp_i}})=U_{\overline{qq_{i}}}$, for $i=0,1,\ldots, n$; in particular, $h(U_G)=U_{\overline{qq_{n}}}$. As above, let $F'\in \textup{Sub}_p(T(Y))$ such that $U_{F'}=h(U_{G'})$, by Proposition \ref{interseccion celdas} there exists an edge $f'\in \textup{edge}(T(Y))$ such that $F'=\overline{qq_n}\cup f'$, in other words, there is a unique vertex $q'=q_{n+1}\in T(X)$ and an edge $f'=\overline{q_jq_{n+1}}\subset T(Y)$ (for some $j=0,1, \ldots, n$) such that $h(U_{\overline{pp'}})=U_{\overline{qq'}}$. Now, it is left to prove that $j=n$. This is a consequence of Lemma \ref{no-interseccion celdas}, indeed, for $i=0, 1,2, \ldots, n-1$, let $G_i=\{p_0\}\cup e_1\cup \ldots \cup e_i $ and $G'=\{p_0\}\cup e_1\cup \ldots \cup e_n\cup e_{n+1}$. Note that $e_{n+1}$ is not incident on $G_i$, then $\overline{U}_{G_i} \cap \overline{U}_{G'} = \emptyset$, this condition is preserved under $h$, that is, $\overline{U}_{\overline{qq_i}} \cap h(\overline{U}_{G'}) =  h(\overline{U}_{G_{i}}) \cap h(\overline{U}_{G'}) = \emptyset$ (for $i=0,1,\ldots,n-1$) this means that the edge $f'=\overline{q_jq_{n+1}}$ is not incident on any of $q_0, \ldots, q_{n-1}$ and therefore it must be incident on $q_n$ and then $h(U_{\overline{pp_i}})=U_{\overline{qq_{i}}}$, for each $i=0,1,\ldots,n+1$. 
\end{proof}

Finally we put all the pieces together to prove one of the main theorems of this paper.

\begin{theorem}\label{ramificacion}
Let $X, Y \in \mathcal{T}$, $p \in R(X)$ and $q \in Y$. If $h : C(p, X) \rightarrow C(q,Y)$ is an homeomorphism then there exists an homeomorphism  $\mathfrak{h}:X\rightarrow Y$ with $\mathfrak{h}(p)=q$.
\end{theorem}

\begin{proof}
We will prove the following stronger statement: there is an isomorphism 
$\mathfrak{h} : X \rightarrow Y$ such that $\mathfrak{h}(p)=q$ (this implies the existence of the desired homeomorphism).
The Proposition \ref{bijection} allows us to define a function $ \mathfrak{h} : V(T(X)) \rightarrow V(T(Y))$ given by $\mathfrak{h}(p') = q'$ if $h(U_{\overline{pp'}})=U_{\overline{qq'}}$. The same construction applied to $h^{-1}$ provides us with its inverse, therefore $\mathfrak{h} : V(T(X)) \rightarrow V(T(Y))$ is a bijection. It is left to prove that it preserves edges. Formally, consider an edge $e=\overline{p'p''}\subset T(X)$ and without lost of generality assume that the path $\overline{pp''}$ consists of edges $e_i=\overline{p_{i-1}p_i}$, $i=1, \ldots, n$, with $p=p_0$, $p'=p_{n-1}$ and $p''=p_n$, then by Proposition \ref{bijection}, we know that the path $\overline{qq''}$ consists of the edges $f_i=\overline{q_{i-1}q_i}$, $i=1, \ldots, n$, with $q_i=\mathfrak{h}(p_i)$ and $q=q_0$, $q'=q_{n-1}$ and $q''=q_{n}$. Hence, there is an edge between $q'=\mathfrak{h}(p')$ and $q''=\mathfrak{h}(p'')$.

Finally we need to prove that $ \mathfrak{h} : V(T(X)) \rightarrow V(T(Y))$ can be extended to all of $V(X)$. Indeed, consider a point $p'\in E(X)$ then there is an edge $e=\overline{p_{n}p'}$ with $p_{n}\in T(X)$. Let $l(p_n)=|\{p' \in E(X):\textrm{there is an edge } e=\overline{p_{n}p'} \subset X\}|$, we will prove $l(p_n) = l(\mathfrak{h}(p_n))$. Consider the path $\overline{pp_n}$ consisting of edges $e_i=\overline{p_{i-1}p_i}$, $i=1, \ldots, n$, with $p=p_0$. We can count $l(p_n)$ as follows:
	\begin{align*}
		l(p_n) &= \textrm{ord}(p_n,X)- \textrm{ord}(p_n, T(X))  \\
		 & = \dim(U_{G'}) - \dim (U_G) + 2 - \textrm{ord}(p_n, T(X)) ,
	\end{align*}
	where in the second line we have used part 3 of Proposition \ref{interseccion celdas} for  $G=\overline{pp_{n-1}}$ and $G'=\overline{pp_n}$. Next, it is already known that $\mathfrak{h} : V(T(X)) \rightarrow V(T(Y))$ is a bijection  preserving edges, which, together with $h$ being a homeomorphism and preserving dimensions implies that 
	\begin{align*}
		l(p_n) 	 & = \dim(U_{G'}) - \dim (U_G) + 2 - \textrm{ord}(p_n, T(X))  \\
		 &= \dim(h(U_{G'})) - \dim (h(U_G)) + 2 - \textrm{ord}(q_n, T(Y))  \\
		 &= \dim(U_{\overline{qq_{n}}}) - \dim (U_{\overline{qq_{n-1}}}) + 2 - \textrm{ord}(q_n, T(Y))  \\
		 &= \textrm{ord}(q_n,Y)- \textrm{ord}(q_n, T(Y))  \\
		 &= l(q_n)  \\
		 &= |\{q' \in E(Y) : \textrm{there is an edge } f=\overline{q_{n}q'} \subset Y \}|,
	\end{align*}
	where $q_{n-1}=\mathfrak{h}(p_{n-1})$ and $q_n=\mathfrak{h}(p_n)$. Thus we have established the existence of a bijection between the sets $\{p' \in E(X):\textrm{there is an edge } e=\overline{p_{n}p'}\} $ and $ \{q' \in E(Y):\textrm{there is an edge } f=\overline{q_{n}q'}\} $, by the construction this bijection preserves incidence relations. Therefore we get a full bijection $\mathfrak{h} : V(X) \rightarrow V(Y)$ preserving the edge relations of the graphs and with $\mathfrak{h}(p)=q$. 
\end{proof}

To complete our discussion now we will extend the previous result to the case when $p\in O(X)\cup E(X)$. In order to do so we need the following lemma.

\begin{lemma}\label{pegandoarcos}
Let $X,Y\in\mathcal{T}$, $p\in R(X)$ and $q\in R(Y)$. If there exists edges $l_{p}$ in $X$ and $l_{q}$ in $Y$, containing $p$ and $q$, respectively, such that $l_{p}\cap E(X)\neq \emptyset$, $l_{q}\cap E(Y)\neq \emptyset$ and $C(p, (X\setminus l_{p})\cup \{p\})$ is homeomorphic to $C(q, (Y\setminus l_{q})\cup \{q\})$, then $C(p,X)$ is homeomorphic to $C(q,Y)$. 
\end{lemma}

\begin{proof}
Let $h:C(p, (X\setminus l_{p})\cup \{p\})\to C(q, (Y\setminus l_{q})\cup \{q\})$ be a homeomorphism. Take $f:[0,1]\to l_{p}$ and $g:[0,1]\to l_{q}$ two homeomorphisms such that $f(0)=p$ and $g(0)=q$. It is easy to see that the map $H: C(p,X) \to C(q,Y)$ given by
$$H(A)=h(A\cap((X\setminus l_{p})\cup\{p\}))\cup g(f^{-1}(A\cap l_{p})), \textrm{ for each $A\in C(p,X)$}$$
is a homeomorphism.
\end{proof}

\begin{theorem}\label{ordinariosfinales}
Let $X\in \mathcal{T}$ and $p\in O(X)\cup E(X)$. Then $(X,p)$ has unique hyperspace $C(p,X)$ in $\mathcal{T}$.
\end{theorem}

\begin{proof}
By Theorems \ref{arcos} and \ref{nodos}, we can suppose that $X\in \widehat{\mathcal{T}}$. Let $p\in O(X)$ and suppose that $Y\in \mathcal{T}$ and $q\in Y$ is such that $C(p,X)$ is homeomorphic to $C(q,Y)$. By Proposition \ref{ordenes}, $\textrm{ord}(X,p)=\textrm{ord}(Y,q)$. Since $X$ can be embedded into the euclidean space $\mathbb{R}^{3}$, suppose that $X$ is a subspace of $\mathbb{R}^{4}$ such that $X\subset \mathbb{R}^{3}\times \{0\}$ and $p=(0,0,0,0)$. Similarly for $Y$. Consider $X^{\prime}=X\cup ((0,0,0)\times [0,1])$ and $Y^{\prime}=Y\cup ((0,0,0)\times [0,1])$. By Lemma \ref{pegandoarcos}, $C(p,X^{\prime})$ is homeomorphic to $C(q,Y^{\prime})$ and using Theorem \ref{ramificacion}, we obtain that there exists an homeomorphism between $X^{\prime}$ and $Y^{\prime}$ sending $p$ to $q$, thus $X$ and $Y$ are homeomorphic. 

The case when $p\in E(X)$ reduces to the case when $p\in O(X)$ by attaching one arc at $p$ and one arc at $q$ becoming these ordinary points.
\end{proof}

Finally we state the main result of this paper which summarizes Theorems \ref{ramificacion} and \ref{ordinariosfinales}.

\begin{theorem}\label{principal}
Let $X\in \mathcal{T}$ and $p\in X$. Then $(X,p)$ has unique hyperspace $C(p,X)$ in the class $\mathcal T$.
\end{theorem}

The next is an application of the previous theorem.

Given a continuum $X$, we consider the hyperspace of $C(C(X))$, $$K(X)=\{C(x,X):x\in X\}.$$
In \cite{CESV(2018)}, it defines in $K(X)$ the following equivalence relation, $C(p,X)\sim C(q,X)$ if and only if $C(p,X)$ is homeomorphic to $C(q,X)$. Given a positive integer $n$ and a continuum $X$, it is say that \textit{$K(X)$ has size $n$} if the quotient $K(X)/\sim$ has cardinality $n$. In \cite{CESV(2018)}, the authors shown that in the class of graphs $X$, the size of $K(X)$ can be different of the homogeneity degree of $X$. As a concequence of Theorem \ref{principal}, in the following result, we obtain a partial solution of \cite[Problem 5.4, p. 49]{CESV(2018)}.

\begin{corollary}
Let $X$ be a tree. Then the size of $K(X)$ is equal to the homogeneity degree of $X$.
\end{corollary}

To finish this paper we present some open problems.

\begin{question}
If $X$ is a dendrite and $p\in X$, has $(X,p)$ unique hyperspace $C(p,X)$ in the class of dendrites?
\end{question}

\begin{question}
If $X$ is a dendroid and $p\in X$, has $(X,p)$ unique hyperspace $C(p,X)$ in the class of dendroids?
\end{question}

\begin{question}
Regarding Example \ref{todoslosarboles}, are there a continuum $X$ and $p\in X$ such that $(X,p)$ has unique hyperspace in the class of continua?
\end{question}

\begin{question}
Are there a continuum $X$ and two points $p,q\in X$ such that $(X,p)$ has unique hyperspace $C(p,X)$, whereas $(X,q)$ does not has it? 
\end{question}

\section{Acknowledgements}

The authors wish to thank to Hugo Villanueva, Eli Vanney Roblero and Sergio Guzm\'an for the fruitful discussions about the topic of this paper.

%%%%%%%%%%%%%%%%%%%%%%%%%%%%%%

\end{document}